\documentclass[10pt]{article}
\usepackage{amsmath,amssymb,amsthm}

\title{Growth of periodic orbits and generalized diagonals for typical triangle billiards}
\author{Dmitri Scheglov\\
University of Oklahoma}

\theoremstyle{plain}
\newtheorem{Lemma}{Lemma}[section] 
\newtheorem{Theorem}{Theorem}[section] 
\newtheorem{Corollary}{Corollary}[section] 
\newtheorem{thm}{Theorem}
\newtheorem{conj}{Conjecture}

\begin{document}

\maketitle
\setlength{\parindent}{0pt}

\begin{abstract}
\noindent
We prove that for any $\epsilon>0$ the growth rate $P_n$ of generalized diagonals or periodic orbits of a typical (in the Lebesgue measure sense) triangle billiard satisfies: $P_n<Ce^{n^{\sqrt{3}-1+\epsilon}}$. This provides an explicit sub-exponential estimate on the triangle billiard complexity and answers a long-standing open question for typical triangles. 
\

This also makes a progress to the problem 3 in the Katok's list of "Five most resistant problems in dynamics". The proof uses essentially new geometric ideas and does not rely on the rational approximations.

\end{abstract}

\section{Introduction}

There are several closely related definitions for a complexity of polygonal billiards. Examples include the growth of periodic orbits, growth of generalized diagonals and orbit complexity. More specifically one can also define directional complexity and position complexity ( see, say [1], [2], [3], [4], [12], [13] ).
\

Each definition measures the growth of orbits, satisfying some special property. It is an easy remark that the growth of periodic orbits cannot exceed the growth of generalized diagonals and in the paper [1] there is a precise relation between growth of generalized diagonals and the growth of all orbits. This implies that any upper estimate on the growth of generalized diagonals would automatically produce an upper estimate on other complexities.
\

\

 In our paper we estimate the growth rate of generalized diagonals.
\

A generalized diagonal is a billiard orbit which connects two vertices. The complexity function $P_n$ is a total number of generalized diagonals of length no greater than $n$. Here by length of the diagonal we mean a discrete length or the number of reflections, but it is well known that the actual geometric length is uniformly proportional to the discrete one.
\

For a polygon with $k$ sides  $P_n\leq k^n$ by trivial combinatorial reasons. 
\

In 1987 Katok [8] proved the following sub-exponential estimate:

\begin{thm}[Katok] For any polygon: $\lim\frac{\ln(P_n)}{n}=0$.
\end{thm}

In 1988-90 Masur [10], [11]  proved more precise estimates for any \textsl{rational-angled} polygon:
\

\begin{thm}[Masur] For any polygon with angles in $\pi\mathbb{Q}$  there are constants $C_1, C_2>0$ such that: $C_1\cdot n^2<P_n<C_2\cdot n^2$.
\end{thm}

 He used an observation that a billiard in the polygon with rational angles is  isomorphic to the geodesic flow on the compact flat surface with a finite number of conical singularities.
\

On such a surface originating from a billiard there is a natural complex structure and moreover a natural choice of holomorphic quadratic differentials which allows to use Teichmuller theory. However for irrational polygons this method can not be applied as the resulting surface is not compact.
\

\

A well-known open problem is to find an explicit sub-exponential estimate for $P_n$ which is considered to be very difficult by many experts. A. Katok in his "Five most resistant problems in dynamics"[7] makes even stronger conjecture:
\

\begin{conj}[Katok] For any polygon $M$ and any $\epsilon>0$: $P_n<C(M, \epsilon)n^{2+\epsilon}$.
\end{conj}

In this strong form the conjecture is quite far from being proven because of the lack of our understanding of the structure of irrational polygonal billiards. However, as we would like to quote A. Katok [5] here: " \textsl{Any} effective sub-exponential estimate ( such as $e^{-T^{3/4}}$, say) for arbitrary polygons would be a major advance." Finding an explicit sub-exponential estimate was a primary motivation for our research. And we would like to thank A. Katok here who mentioned  the importance of this problem several times during about 6 years to us, which definitely gave an extra-motivation.
\

\

One of the reasons why it is so difficult to analyse $P_n$ is that it is a purely discrete counting of different orbits and so it does not take into account any orbit structure such as density of orbits in a particular angular region or distribution of orbits with respect to the natural invariant measure. 
\

This distinguishes the complexity growth from other dynamical characteristis. For example ergodicity of some irrational polygons was proven by applying approximating arguments and moreover the result of Vorobets [14] explicitely describes some well-approximated ergodic polygons. 

\

For completeness of the exposition we would like to briefly discuss the key ideas of the original paper by Katok [8].
\

 He considers a topological subshift on $k$ symbols, naturally associated to the billiard in $k$ - gon and proves that any ergodic invariant measure is supported on the subset, generated by the images of actual billiard orbits.
\

Then he proves that metric entropy of any such measure is equal to zero and then by variational principle it implies that the topological entropy is also zero. As the symbolic cylinder growth in this setting can be reformulated in terms of $P_n$, the fact that topological entropy is zero completes the proof.
\

\

Even though the proof is elegant it has several non-explicit steps which make it hard to extract more precise information about $P_n$ than sub-exponential growth. First of all it uses ergodic invariant measures and a variational principle and it is not clear how to make this abstract argument constructive. 
\

And the second point is that the topological entropy can only distinguish exponential growth and does not 'feel' any sub-exponential effects, where hypothetically some kind of \textsl{ slow entropy} is required to extract non-trivial information.
\

\

Our approach is more geometric and combinatorial and not ergodic-theoretic. The aim of the paper is to prove the following theorem:
\

\begin{thm}[Explicit sub-exponential estimate] For a typical triangle and any $\epsilon>0$ there is a constant $C>0$ such that: $P_n<Ce^{n^{\sqrt{3}-1+\epsilon}}$.
\end{thm}

\textbf{Acknowledgements.} We would like to thank Dr. John Albert and Dr. Christian Remling for useful discussions and Dr. Andrei Gogolev for useful discussions and constant encouragement along the progress in the work.
We would also like to thank Dr. Anatole Katok and Dr. Federico Rodriguez-Hertz for reading  the paper, encouragement and a serious help when the gap was found in the first version of the paper.
\

We would also like to express our special acknowledgements to Dr. Serge Troubetzkoy who very carefully read the paper and made several very important remarks which helped to improve its structure and the quality of exposition, and to Dr. Eugene Gutkin for his several important remarks, constant interest to the paper and several useful conversations about billiard dynamics.
\

Without support of these mathematicians this paper would not appear.
\section{Interval partitions}

We consider a given triangle, a fixed vertex and the corresponding anglular segment located at the vertex, which we naturally associate with an interval $I$ using the angular distance on it. In this setting points on the interval correspond to rays emanating from the vertex.
\

\

Now let us create a decreasing sequence of finite indexed partitions $\xi_n $ of $I$ on subintervals as follows. $\xi_0=I$ a  trivial partition with one element. Cutting points of partition $\xi_n$ are those corresponding to the generalized diagonals of length no greater than $n$.
\

The following two properties immediately follow from this construction:
\

\

$1)$ Inside each interval of the partition $\xi_n$ there is at most one point of the partition $\xi_{n+1}$.
\

$2)$ The sequence $\xi_n$ converges to the partition on points. In the other words the union of all cutting points is dense in $I$.
\

\

By construction the number $P_n$ of generalized diagonals is exactly the number of cutting points of $\xi_n$ and each cutting point has an index, namely the length of the corresponding generalized diagonal. 

\subsection{Points in a good position}

Consider a sequence of partitions $\xi_n$ and 3 indexed points $x_p$, $x_q$, $x_r$ as cutting points of corresponding partitions, such that $p<q<r$.
\

\

\textbf{ Definition.} The points $x_p$, $x_q$, $x_r$, $p<q<r$ are in a \textsl{good position} if:
\

$1)$ In the  interval, bounded by $x_p$ and $x_q$ there are no points with index $<r+1$ except $x_r$.
\

$2)$  The point $x_r$ lies between $x_p$ and $x_q$  
\

\

The picture below shows points in a good position.

\

\unitlength=1.00mm
\special{em:linewidth 0.4pt}
\linethickness{0.4pt}
\begin{picture}(91.67,14.00)
\put(6.00,10.67){\line(1,0){85.67}}
\put(28.00,10.67){\circle*{2.00}}
\put(49.33,10.67){\circle*{2.00}}
\put(73.33,10.67){\circle*{2.11}}
\put(28.00,14.00){\makebox(0,0)[cc]{p}}
\put(49.33,14.00){\makebox(0,0)[cc]{r}}
\put(73.33,14.00){\makebox(0,0)[cc]{q}}
\end{picture}

\unitlength=1.00mm
\special{em:linewidth 0.4pt}
\linethickness{0.4pt}
\begin{picture}(91.67,14.00)
\put(6.00,10.67){\line(1,0){85.67}}
\put(28.00,10.67){\circle*{2.00}}
\put(49.33,10.67){\circle*{2.00}}
\put(73.33,10.67){\circle*{2.11}}
\put(28.00,14.00){\makebox(0,0)[cc]{q}}
\put(49.33,14.00){\makebox(0,0)[cc]{r}}
\put(73.33,14.00){\makebox(0,0)[cc]{p}}
\end{picture}

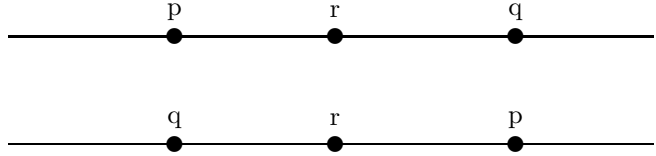
Pic 1. Points in good position.
\

\

Below we present an example of a configuration of points $x_p$, $x_q$, $x_r$, $p<q<r$ which are NOT in good position.

\

\unitlength=1.00mm
\special{em:linewidth 0.4pt}
\linethickness{0.4pt}
\begin{picture}(91.67,14.00)
\put(6.00,10.67){\line(1,0){85.67}}
\put(28.00,10.67){\circle*{2.00}}
\put(49.33,10.67){\circle*{2.00}}
\put(73.33,10.67){\circle*{2.11}}
\put(28.00,14.00){\makebox(0,0)[cc]{r}}
\put(49.33,14.00){\makebox(0,0)[cc]{p}}
\put(73.33,14.00){\makebox(0,0)[cc]{q}}
\end{picture}
\

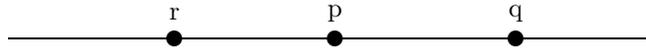
Pic 2. Points NOT in good position.

\begin{Lemma}
 Consider an interval $J$ belonging to the partition $\xi_n$ and a finite sequence of partitions $\xi_{n+1},\ldots, \xi_{n+c}$. Let $S$ be the set of all cutting points of $\xi_{n+c}$ inside $J$. Assume that the cardinality $|S|\geq 4+2c$.
\

Then there exist points $x_p, x_q, x_r\in S$ in a good position.
\end{Lemma}

\begin{proof}
 Assume that a partition $\xi_{n+m}$ has at least 3 points : $x_{i_1}< x_{i_2}<\ldots< x_{i_p}$ inside $J$. Then the partition $\xi_{n+m+1}$ may only have at most two more points inside $J$: $x_{n+m+1}, y_{n+m+1}$ :  $x_{n+m+1}<x_{i_1}$ and $x_{i_p}<y_{n+m+1}$ without producing a triple in good position. So the set $S$ without a triple in good position has at most $3+2c$ points.
\end{proof}

\subsection{ Existence of close points in a good position}

\begin{Lemma} Consider a finite sequence of partitions $\xi_{n}, \xi_{n+1},\ldots,\xi_{n+c}$. Assume that $P_{n+c}\geq (4+2c)P_n$. We also assume that 
$c\geq 4$ and $e^c<P_n$. Then there exist 3 points in good position with indices from the range $[n+1,n+c]$ and  $e^c/P_n$ - close to each other.
\end{Lemma}

\begin{proof} The number of cutting points of $\xi_{n+c}$ inside each interval of the partition $\xi_n$ is bounded by $2^c$. Let $x$ be the number of intervals of the partition $\xi_n$ which have at least $4+2c$ points of $\xi_{n+c}$ inside and $y$ be the number of intervals with less than $4+2c$ points of $\xi_{n+c}$  inside.
We then have two obvious relations:
\

\

1) $x+y=P_n+1$
\

\

2) $2^cx+(3+2c)y\geq (4+2c)P_n$ 
\

\

from which $x\geq\frac{P_n-3-2c}{2^c-3-2c}>P_n/e^c$ follows.
\

\

By Lemma 2.1 each interval of $\xi_n$ containing at least $4+2c$ points also contains 3 points in a good position with indices in the range $[n+1,\ldots, n+c]$.
\

As the corresponding intervals do not intersect and their total number is at least $x$ and they all are contained in the interval $[0, 1]$, the estimate on $x$ completes the proof.
\end{proof}

\subsection{ Index interval estimates}

In this section we estimate the constant $c$ which gives a range for indices of points in a good position. We introduce constants $\mu, \epsilon, \gamma>0$ to be chosen later.  Along the way we specify the assumptions  to be satisfied for these constants (we use \textbf{bold} font for that) and in the end we summarize all the assumptions and choose the particular values for $\mu, \epsilon, \gamma$.
\

We now introduce the following functions: $\phi(n)=n^{\mu}$, $c(n)=n^{\epsilon}$, $k(n)=n^{\gamma}$.
\

\begin{Lemma} Assume a sequence of partitions $\xi_n$ satisfies inequality $P_n\geq e^{\phi(n)}$ for any $n$ large enough. Then for an appropriate choice of constants $\mu, \gamma, \epsilon$ and any $n$ large enough there exists an integer $N$: $n<N<n+k(n)c(n) $ such that
 $P_{N+c(n)}/P_{N}\geq 4+2c(n)$.
\end{Lemma}

\begin{proof} Divide interval $[n+1,n+k(n)c(n)]$ into $k(n)$ intervals $I_1=[n+1, n+c(n)]$, $I_2=[n+c(n)+1, n+2c(n)]$,$\ldots$, $I_{k(n)}=[n+(k(n)-1)c(n)+1, n+k(n)c(n)]$.

We prove by contradiction that for one of the intervals $I_l$:  $P_{N+c(n)}/P_N\geq 4+2c$, where $N=n+lc(n)$.
\

Assume that for all the intervals $I_l$: $P_{n+(l+1)c(n)}/P_{n+lc(n)}<4+2c(n)$. Then $P_{n+k(n)c(n)}/P_{n}<(4+2c(n))^{k(n)}$.
\

\

On the other hand $P_{n+k(n)c(n)}/P_{n}>e^{\phi(n+k(n)c(n))}/e^n$ which implies:
\

\

$\phi(n+k(n)c(n))-n<k(n)\ln(4+2c(n))$.
\

\

We assume that \boldmath$\gamma\leq 1$ and $\gamma+\epsilon>1$ \unboldmath
so $k(n)c(n)$ has degree higher than 1. Then comparing the highest degrees on both sides implies
 $(\gamma+\epsilon)\mu\leq 1$

which means that we get a contradiction under assumption \boldmath$(\gamma+\epsilon)\mu> 1$. \unboldmath
\end{proof}

\section{Combinatorial geometry of orbits}

\subsection{ Unfolding of a billiard trajectory}
\

We would like to remind of a useful unfolding construction associated to any polygonal billiard.[9]
\

We fix a polygon on the plane and consider a time moment when a particular billiard orbit hits a polygon side. Then instead of reflecting the orbit we continue it as a straight line and reflect the polygon along the line.
\

As we continue this process indefinitely the sequence of polygons obtained this way is called unfolding of the polygon along the orbit.
\

The picture below illustrates the unfolding of a triangle along an orbit.
\

\unitlength=1.00mm
\special{em:linewidth 0.4pt}
\linethickness{0.4pt}
\begin{picture}(94.67,40.33)
\put(9.00,3.67){\line(-1,4){6.67}}
\put(2.33,29.67){\line(2,3){6.33}}
\put(8.66,39.34){\line(2,-3){6.67}}
\put(15.33,29.34){\line(-1,-4){6.00}}
\put(9.00,3.67){\line(1,1){17.33}}
\put(26.33,21.00){\line(0,1){13.00}}
\put(26.33,34.00){\line(-5,-2){11.33}}
\put(4.00,23.34){\vector(1,0){90.67}}
\put(26.33,34.00){\line(6,-1){11.00}}
\put(37.33,32.34){\line(1,-3){8.00}}
\put(45.33,8.34){\line(-3,2){19.00}}
\put(9.00,4.00){\line(0,1){35.00}}
\put(9.00,4.33){\line(3,5){18.00}}
\put(45.33,8.33){\line(-3,4){19.00}}
\put(37.33,32.33){\line(1,1){8.00}}
\put(45.33,40.33){\line(0,-1){31.67}}
\end{picture}
\

Pic 3. Triangle unfolding.
\

\

For a given triangle the shape obtained from a triangle by reflection about one side is called a \textsl{kite}. It is clear that for any triangle unfolding there is an associated kite unfolding. We will use both unfoldings having in mind the natural correspondence between them.
\

The next picture shows the corresponding kite unfolding.
\

\unitlength=1.00mm
\special{em:linewidth 0.4pt}
\linethickness{0.4pt}
\begin{picture}(94.67,42.67)
\put(9.00,3.67){\line(-1,4){6.67}}
\put(2.33,29.67){\line(2,3){6.33}}
\put(8.66,39.34){\line(2,-3){6.67}}
\put(15.33,29.34){\line(-1,-4){6.00}}
\put(9.00,3.67){\line(1,1){17.33}}
\put(26.33,21.00){\line(0,1){13.00}}
\put(26.33,34.00){\line(-5,-2){11.33}}
\put(4.00,23.34){\vector(1,0){90.67}}
\put(26.33,34.00){\line(6,-1){11.00}}
\put(37.33,32.34){\line(1,-3){8.00}}
\put(45.33,8.34){\line(1,5){5.33}}
\put(45.33,8.34){\line(-3,2){19.00}}
\put(50.66,35.00){\line(-1,1){7.67}}
\put(43.33,42.67){\line(-3,-5){6.33}}
\end{picture}
\

Pic 4. Kite unfolding
\

\

As we see from the picture above, any kite unfolding along the orbit consists of consecutive rotations of the kite along one of the two kite vertices, corresponding to the angles $\alpha$ and $\beta$ of the original triangle with angles correspondingly $ 2\alpha$ and $2\beta$.
\

We now assume that a kite is located in the standard Euclidean $xy$ coordinate plane and introduce several notations. 
\

\

The \textsl{$\alpha$-vertex} and \textsl{$\beta$- vertex} are kite vertices corresponding to the angles $2\alpha$ and $2\beta$. Two other vertices are called \textsl{ side vertices}.
\

\

 The \textsl{kite diagonal} is a vector going from the $\alpha$-vertex to the $\beta$-vertex.
\

\

The \textsl{kite angle} is a \textsl{counterclockwise} angle between $x$-axis and the kite diagonal.
\

\

On pic.5 $A$ and $B$ are $\alpha$ and $\beta$ vertices correspondingly, vector $\overrightarrow{AB}$ is a kite diagonal, $C$ and $D$ are side vertices.
Pic.6. shows a kite in standard position on the $xy$ plane.
\

\

 Note that any unfolding of $K$ is uniquely characterized by the sequence of angles $\pm 2\alpha$ or $\pm 2\beta$ depending on the kite vertex we rotate about and the direction of rotation. Such a sequence of angles is called the \textsl{combinatorics} of a kite unfolding.
\

\

\unitlength=1.00mm
\special{em:linewidth 0.4pt}
\linethickness{0.4pt}
\begin{picture}(112.33,67.34)
\put(2.00,19.00){\line(1,0){110.33}}
\put(9.34,4.67){\line(0,1){62.67}}
\put(62.34,29.34){\line(1,1){34.33}}
\put(96.67,63.67){\line(-1,0){18.33}}
\put(78.34,63.67){\line(-1,-2){17.67}}
\put(96.67,63.67){\line(-1,-1){36.00}}
\put(60.67,27.67){\line(1,1){36.00}}
\put(60.67,28.00){\line(5,2){36.00}}
\put(96.67,42.33){\line(0,1){21.33}}
\put(61.00,28.33){\circle*{1.33}}
\put(96.67,63.67){\circle*{2.11}}
\put(78.00,63.67){\circle*{2.00}}
\put(96.67,42.33){\circle*{2.00}}
\put(61.00,28.33){\circle*{2.00}}
\put(56.67,28.00){\makebox(0,0)[cc]{$A$}}
\put(100.33,63.67){\makebox(0,0)[cc]{$B$}}
\put(73.67,63.67){\makebox(0,0)[cc]{$C$}}
\put(100.33,42.33){\makebox(0,0)[cc]{$D$}}
\put(110.33,15.33){\makebox(0,0)[cc]{$x$}}
\put(6.00,66.00){\makebox(0,0)[cc]{$y$}}
\put(68.33,38.67){\makebox(0,0)[cc]{$\alpha$}}
\put(71.33,34.67){\makebox(0,0)[cc]{$\alpha$}}
\put(89.00,60.67){\makebox(0,0)[cc]{$\beta$}}
\put(94.00,56.67){\makebox(0,0)[cc]{$\beta$}}
\end{picture}
\

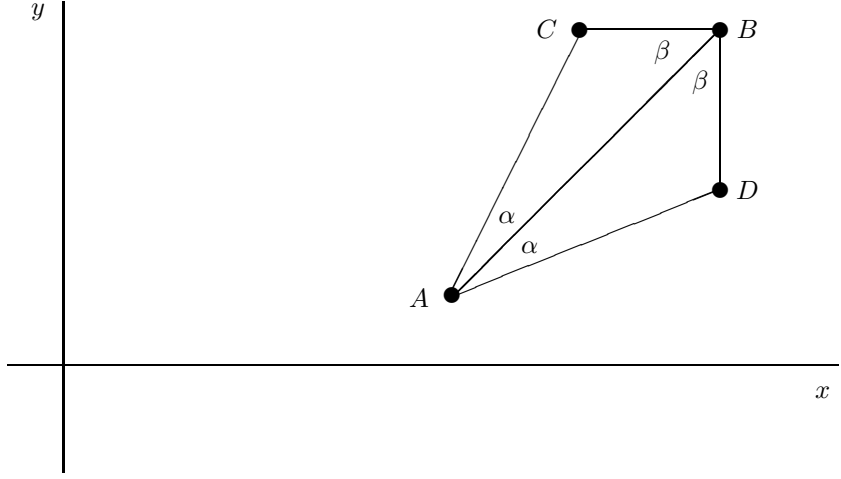
Pic 5. A kite on the $xy$ coordinate plane.
\

\

\unitlength=0.5mm
\special{em:linewidth 0.4pt}
\linethickness{0.4pt}
\begin{picture}(129.33,114.67)
\put(9.00,52.00){\line(1,-1){29.67}}
\put(9.33,51.67){\line(1,1){29.33}}
\put(38.66,81.00){\line(1,-2){14.67}}
\put(53.33,51.67){\line(-1,-2){14.67}}
\put(9.33,52.00){\line(1,0){117.67}}
\put(9.33,12.67){\line(0,1){99.33}}
\put(15.00,54.00){\makebox(0,0)[cc]{$\alpha$}}
\put(15.00,49.67){\makebox(0,0)[cc]{$\alpha$}}
\put(49.00,54.33){\makebox(0,0)[cc]{$\beta$}}
\put(49.00,49.00){\makebox(0,0)[cc]{$\beta$}}
\put(129.33,47.33){\makebox(0,0)[cc]{$x$}}
\put(3.33,114.67){\makebox(0,0)[cc]{$y$}}
\end{picture}
\

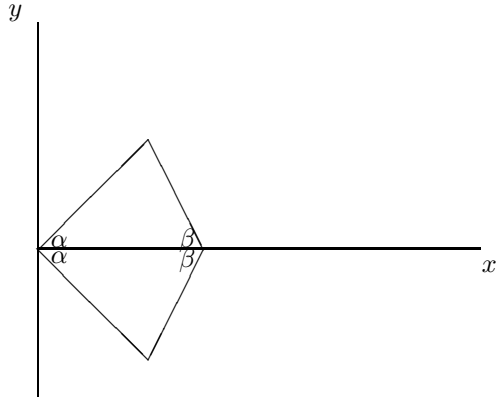
Pic 6. A kite in the standard position on the $xy$ plane.

\begin{Lemma} Assume a kite $K$ with a diagonal length 1 is in standard position and a kite $K'$ is obtained from $K$ by means of a particular combinatorics of length $n$. Let $x^{\alpha}_n$, $y^{\alpha}_n$ and $x^{\beta}_n$, $y^{\beta}_n$ be the coordinates of $\alpha$ and $\beta$  vertices of $K'$ and let $x_n$, $y_n$ be the coordinates of either of the two side vertices of $K'$. Then:
\

$1)$ $x^{\alpha}_n$, $y^{\alpha}_n$, $x^{\beta}_n$, $y^{\beta}_n$ are represented by trigonometric polynomials of angles $\alpha$, $\beta$ with integer coefficients, depending only on the combinatorics and of degree at most $2n-2$.
\

$2)$ $x_n=P_{2n}(\alpha, \beta)+ \frac{\sin (\beta)}{\sin(\alpha+\beta)}\cdot \cos(m\alpha+l\beta)$
\

$y_n=Q_{2n}(\alpha, \beta)+ \frac{\sin (\beta)}{\sin(\alpha+\beta)}\cdot \sin(m\alpha+l\beta)$,
\

where $P_{2n}(\alpha,\beta)$, $Q_{2n}(\alpha, \beta)$ are trigonometric polynomials with integer coefficients of degree at most $2n-2$ and $|m|+|l|\leq 2n-1$.
\end{Lemma}

\begin{proof} The proof of the first statement goes by an easy induction on $n$. For n=1 the statement is trivial. If $\phi_n$ is the kite angle on the $n$-th step and on the $n+1$-th step we rotate, say, about $\alpha$-vertex, then $\phi_{n+1}=\phi_n\pm 2\alpha$ and $x^{\alpha}_{n+1}=x^{\alpha}_n$,  $y^{\alpha}_{n+1}=y^{\alpha}_n$, 
$x^{\beta}_{n+1}=x^{\alpha}_n+\cos(\phi_{n+1})$,  $y^{\beta}_{n+1}=y^{\alpha}_{n}+\sin(\phi_{n+1})$.
\

The case when we rotate about $\beta$-vertex is entirely analogous. This completes the induction step.
\

\

The second statement easily follows from the first one by noticing that the length of the side, adjacent to the $\alpha$-vertex is $\frac{\sin (\beta)}{\sin(\alpha+\beta)}$ and so if $\phi_n$ is the kite angle of $K'$ then $x_n=x^{\alpha}_n+\frac{\sin (\beta)}{\sin(\alpha+\beta)}\cdot \cos(\phi_n\pm\alpha)$ and 
 $y_n=y^{\alpha}_n+\frac{\sin (\beta)}{\sin(\alpha+\beta)}\cdot \sin(\phi_n\pm\alpha)$.
\end{proof}

\section{Complexity estimate}

In this chapter we assume that the triangle has a fixed side of length 1 and adjacent angles are acute and for some arbitrarily small parameter $\delta>0$ satisfy:
\

$\alpha>\delta$, $\beta>\delta$, $\alpha+\beta<\pi-\delta$.
\

\

This condition guarantees that there are constants $D_{\delta}, N_{\delta}>0$ such that for any billiard orbit with $n$ reflections :
 $L(n)/D_{\delta}<n<L(n)D_{\delta}$, for $n>N_{\delta}$,
 where $L(n)$ is a geometric length of the orbit.
As $\delta$ can be chosen arbitrarily small the conclusion of the theorem would hold for a full space of triangles.
\

\begin{Theorem}[Subsequence complexity]  For a full measure set of triangles and any $\epsilon>0$: $\lim\inf P_n\cdot e^{-n^{\sqrt{3}-1+\epsilon}}<\infty$.
\end{Theorem}

\begin{proof}
We consider a triangle and a vertex and fix  $n$ large enough. Assume that $P_n\geq e^{\phi(n)}$. Then by Lemma 2.3. we find $N$: $n<N<n+k(n)c(n)$ such that $P_{N+c(n)}/P_{N}\geq 4+2c(n)$.
\

To use Lemma 2.2. we need to make sure that $e^{c(n)}<P_{N}$ which is true if $c(n)<n^{\mu}$ which in turn is satisfied if \boldmath$\epsilon<\mu$.
\unboldmath
\

Now by Lemma 2.2. there are points $x_p, x_q, x_r$ in a good position, where $p<q<r$; $p, q, r\in [N, N+c(n)]$ and with pairwise distances bounded by $e^{c(n)}/P_n$.
\

\

The good position of points guarantees that there exists a direction $z$ which has the same unfolding combinatorics at times $p, q, r$ as corresponding directions $x_p$, $x_q$ and $x_r$ and such that the directions $x_p$ and $x_r$ lie on the different side from the direction $z$ then the direction $x_q$. It is achieved by taking any direction $z$ lying between $x_q$ and $x_r$.
\

The picture below illustrates this observation.
\

\

\unitlength=0.50mm
\special{em:linewidth 0.4pt}
\linethickness{0.4pt}
\begin{picture}(288.67,79.34)
\put(18.00,64.00){\line(-1,-1){16.67}}
\put(1.34,47.34){\line(2,-5){16.67}}
\put(18.00,5.67){\line(0,1){58.33}}
\put(1.34,46.67){\line(1,0){287.33}}
\put(107.34,72.67){\line(1,-1){39.33}}
\put(146.67,33.34){\line(0,1){29.33}}
\put(146.67,62.67){\line(-4,1){39.33}}
\put(200.00,40.67){\line(-1,1){17.33}}
\put(182.67,58.00){\line(-1,-1){34.67}}
\put(148.00,23.34){\line(3,1){52.00}}
\put(241.34,68.00){\line(0,-1){59.33}}
\put(1.34,46.67){\line(1,0){254.67}}
\put(241.34,8.67){\line(1,2){23.33}}
\put(264.67,55.34){\line(-1,1){24.00}}
\put(241.34,58.67){\line(0,1){20.00}}
\put(241.34,48.67){\line(0,1){20.67}}
\put(146.67,63.34){\circle*{2.98}}
\put(200.00,40.67){\circle*{2.98}}
\put(264.00,55.34){\circle*{2.67}}
\put(150.00,65.34){\makebox(0,0)[cc]{$p$}}
\put(205.34,38.00){\makebox(0,0)[cc]{$q$}}
\put(270.00,58.00){\makebox(0,0)[cc]{$r$}}
\end{picture}

\

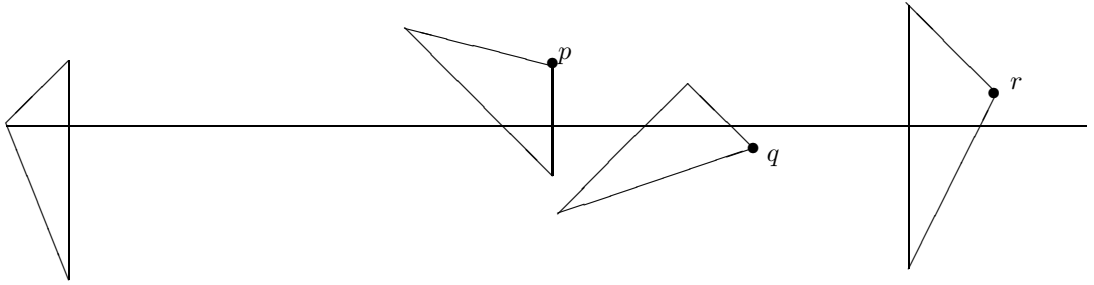
Pic 7. For points $x_p, x_q, x_r$ in a good position, there is a direction with the same unfolding combinatorics at times $p, q, r$.
\

\

As directions $x_p, x_q, x_r$ are generalized diagonals, they hit triangle vertices at times $p, q, r$. For simplicity we will denote corresponding  vertices as $P, Q, R$.
\

Since $r<n+k(n)c(n)$, and the angular distances are bounded by $e^{c(n)}/P_n$ we obtain that the distances from points $P, Q, R$ to the $z$-trajectory are bounded by $d= D_{\delta}(n+k(n)c(n))e^{c(n)}/P_n<e^{-an^{\mu}}$ for some constant $a>0$.
\

\

We now look more carefully at the piece of $z$-trajectory between points $P$ and $R$ of length $c(n)$. We complete each triangle to a kite so that a chosen triangle side of length 1 corresponds to the kite diagonal. Now the triangle unfolding of the $z$-trajectory from $P$ to $R$ corresponds to the kite unfolding.
\

Rotate the picture so that the $P$-kite is in standard position.
\

\

\unitlength=0.50mm
\special{em:linewidth 0.4pt}
\linethickness{0.4pt}
\begin{picture}(288.67,90.67)
\put(7.34,23.00){\line(5,-3){25.33}}
\put(32.67,7.67){\line(3,5){9.00}}
\put(41.67,22.67){\line(-3,5){9.33}}
\put(32.34,38.00){\line(-5,-3){25.00}}
\put(42.00,22.67){\line(1,1){14.67}}
\put(56.67,37.33){\line(-2,5){13.33}}
\put(43.34,70.00){\line(-1,-3){10.67}}
\put(56.67,37.33){\line(5,-3){22.67}}
\put(79.34,24.00){\line(-3,-2){22.67}}
\put(56.67,8.67){\line(-1,1){14.67}}
\put(136.00,41.33){\line(4,3){26.00}}
\put(162.00,60.67){\line(1,-1){14.00}}
\put(176.00,46.67){\line(-1,-2){9.33}}
\put(166.67,28.67){\line(-5,2){30.00}}
\put(136.67,40.67){\line(1,6){6.00}}
\put(142.67,76.67){\line(4,1){19.33}}
\put(162.00,81.33){\line(0,-1){20.67}}
\put(162.00,60.67){\line(1,2){14.67}}
\put(176.67,90.00){\line(2,-5){11.33}}
\put(188.00,62.00){\line(-4,-5){12.00}}
\put(136.67,40.67){\circle*{2.67}}
\put(8.67,22.67){\circle*{2.67}}
\put(6.00,26.67){\makebox(0,0)[cc]{$P$}}
\put(134.00,34.67){\makebox(0,0)[cc]{$Q$}}
\put(4.00,17.33){\line(4,1){284.67}}
\put(256.00,85.33){\line(-5,-1){28.67}}
\put(227.33,79.33){\line(-1,-4){4.00}}
\put(223.33,62.00){\line(4,-1){13.33}}
\put(236.67,58.67){\line(2,3){18.00}}
\put(258.67,90.67){\makebox(0,0)[cc]{$R$}}
\put(254.67,85.33){\circle*{2.98}}
\end{picture}

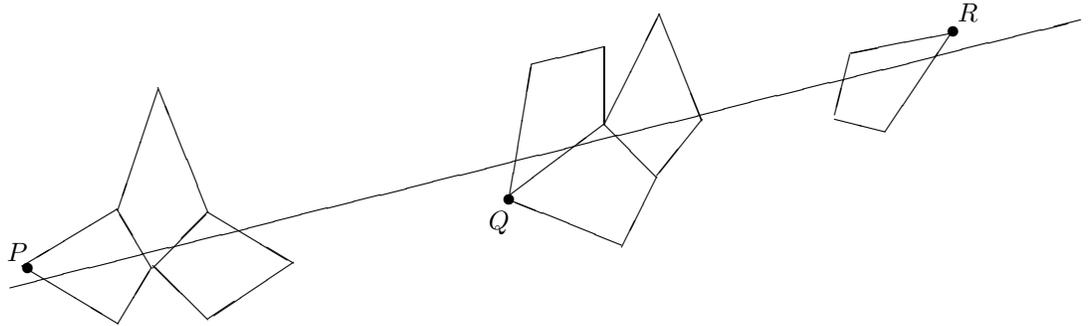
Pic.8 A $z$-trajectory from $P$ to $R$ starting in standard position.
\

\

By Lemma 3.1 the $x$ and $y$  coordinates of the points $P, Q, R$ can be represented as $\frac {L(\alpha,\beta)}{\sin(\alpha+\beta)}$, where $L(\alpha, \beta)$ is a trigonometric polynomial with integer coefficients of degree at most $2c(n)$ and so the area of the triangle $PQR$ can be represented as 
$ \frac{A(\alpha, \beta)}{\sin^2(\alpha+\beta)}$, where $A(\alpha, \beta)$ is a trigonometric polynomial with integer coefficients of degree at most $4c(n)=4n^{\epsilon}$.
\

The estimates on $d$ imply:
 $|A(\alpha, \beta)|\leq D_{\delta}\sin^2(\alpha+\beta)c(n)d<e^{-bn^{\mu}}$, for some $b>0$.
\

\

\textbf{Note} that $A(\alpha,\beta)\neq 0$ because $p<q<r$ and the points $P$ and $R$ lie on a \textbf{\textsl{different}} side of $z$ direction than the point $Q$. This is an extremely important observation and it is this particular point of the proof which motivates our definition of a good position.
\

\

Let $\mathcal{F}_n$ be a set of trigonometric polynomials corresponding to unfolding combinatorics of length $c(n)=4n^{\epsilon}$.  Any polynomial is uniquely determined by the combinatorics and a choice of vertices. It implies that the cardinality  $|\mathcal{F}_n|< e^{tn^{\epsilon}}$, for some $t>0$.
We are now going to estimate the measure of  the following set of triangle angles: 
${\mathcal{B}}_n=\lbrace (\alpha, \beta)| \exists A\in{\mathcal{F}}_n :|A(\alpha,\beta)|<e^{-bn^{\mu}}\rbrace $.
\

The key point we use here is that a non-trivial trigonometric polynomial with integer coefficients can not be small on the set of large measure.
To make this point more precise we refer to the very useful theorem by Kaloshin and Rodnianski [6]  which can be formulated as follows:

\begin{Theorem}[Kaloshin, Rodnianski] There exist universal constants $R, c>0$ such that any non-zero trigonometric polynomial with integer coefficients $P$ in variables $\alpha, \beta, \gamma\in [0, 2\pi]$  of degree at most $m$ satisfies :
\

\

$Leb\lbrace (\alpha, \beta, \gamma): |P(\alpha, \beta, \gamma)|<e^{-Rm^{2}}\rbrace<e^{-cm}$.
\end{Theorem}

Any trigonometric polinomial $P(\alpha, \beta)$ in 2 variables can be considered as a polynomial $P(\alpha, \beta, \gamma)$ of three variables of the same degree, where the variable $\gamma$ is not present. 
Moreover any level set for $P$ in variables $\alpha, \beta, \gamma$ is obtained from the level set for $P$ in variables $\alpha, \beta$ by multiplying on segment $[0, 2\pi]$ in variable $\gamma$. Then an easy use of the Fubini theorem implies the following corollary:
\

\begin{Corollary} There exist universal positive constants $R, c$ such that any non-zero trigonometric polynomial with integer coefficients $P$ in variables $\alpha, \beta\in [0, 2\pi]$  of degree at most $m$ satisfies the following inequality:
\

\

$Leb\lbrace (\alpha, \beta): |P(\alpha, \beta)|<e^{-Rm^{2}}\rbrace<e^{-cm}$.
\end{Corollary}

We now pick $A\in{\mathcal{F}}_n$ and take $m=Fn^{\epsilon}$, where $F>4$ is a large enough constant to be chosen later.
\

By corollary 4.1:  $Leb\lbrace (\alpha, \beta): |A(\alpha, \beta)|<e^{-RF^{2}n^{2\epsilon}}\rbrace<e^{-cFn^{\epsilon}}$
\

We need now is to show that for large enough $n$: $e^{-bn^{\mu}}<e^{-RF^2n^{2\epsilon}}$ which
by comparing the highest degrees is true if \boldmath$2\epsilon<\mu$ \unboldmath.
\

 As $A\in{\mathcal{F}}_n$ it implies $Leb({\mathcal{B}}_n)<  e^{-cFn^{\epsilon}}|{\mathcal{F}}_n|<e^{(t-cF)n^{\epsilon}}$ and so if we choose $F>t/c$ then
$\sum Leb({\mathcal{B}}_n)<\infty$.
\

\

From the argument above it follows that under assumption $P_n>e^{\phi(n)}$ for large enough $n$ the pair $(\alpha, \beta)\in {\mathcal{B}}_n$
 and a standard Borel-Cantelly argument completes the proof for the appropriate choice of $\mu, \epsilon, \gamma$.

\subsection { Choice of constants}

Here we summarize all the assumptions on constants $\mu, \epsilon, \gamma>0$ which we met in the proof and choose a minimal $\mu$ satisfying them. 
\

\

$\begin{cases}\gamma\leq 1\\ \gamma+\epsilon>1\\  (\gamma+\epsilon)\mu>1 \\ \epsilon<\mu  \\  2\epsilon<\mu \end{cases}$
\

\

It is clear that we may take $\gamma =1$ and then the problem reduces to minimizing $\mu$ satisfying:
\

\

$\begin{cases} (1+\epsilon)\mu>1 \\  2\epsilon<\mu \end{cases}$
\

\

Taking the extreme case we get: $(1+\mu/2)\mu=1$,
\

which in turn implies that all the conditions above can be satisfied for any 
\

$\mu>\sqrt{ 3}-1$.
\
\end{proof}

\subsection{ Complexity estimate.}

In this section we use theorem 4.1 to get a global complexity estimate.
\

Let us fix arbitrary $\mu>\sqrt{3}-1$. By the theorem 4.1 for any triangle $\Delta$ from a full measure set $X$ of triangles and any vertex there exists a monotone sequence of times $n_i$ characterized by the property: $P_{n_i}<e^{n_i^{\mu}}$. Our aim now is to estimate the gap $n_{i+1}-n_i$.

\begin{Theorem} For any triangle $\Delta\in X$ and any $\epsilon>0$ under assumptions above for all $i$ large enough: $n_{i+1}-n_i<n_i^{1+\epsilon}$.
\end{Theorem}

\begin{proof}  In the proof we repeat the previous arguments with mild changes. First we introduce the following notations: 
$\phi(n)=n^{\mu}, k(n)=n^{\mu}, c(n)=n^{1-\mu+\epsilon}$.

\begin{Lemma} Assume that for a fixed triangle and for some $i$ large enough  $n_{i+1}-n_{i}>n_i^{1+\epsilon}$. Then there exists an integer $N$:
 $n_i<N<n_{i+1} $ such that
 $P_{N+c(n_i)}/P_{N}\geq 4+2c(n_i)$.
\end{Lemma}

\begin{proof} From the Lemma assumptions it follows that $n_i+k(n_i)c(n_i)<n_{i+1}$. Divide the interval $[n_i+1, n_i+k(n_i)c(n_i)]$ into $k(n_i)$ intervals of length $c(n_i)$: 
\

$I_1=[n_i+1, n_i+c(n_i)]$, $I_2=[n_i+c(n_i)+1, n_i+2c(n_i)]$,$\ldots$, $I_{k(n_i)}=[n_i+(k(n_i)-1)c(n_i)+1, n_i+k(n_i)c(n_i)]$.
\

We prove by contradiction that for one of the intervals $I_l$:  $P_{N+c(n_i)}/P_N\geq 4+2c(n_i)$, where $N=n_i+lc(n_i)$.
\

Assume that for all the intervals $I_l$: $P_{n_i+(l+1)c(n_i)}/P_{n_i+lc(n_i)}<4+2c(n_i)$. Then $P_{n_i+k(n_i)c(n_i)}/P_{n_i}<(4+2c(n_i))^{k(n_i)}$.
\

On the other hand $P_{n_i+k(n_i)c(n_i)}/P_{n_i}>e^{\phi(n_i+k(n_i)c(n_i))}/e^{\phi(n_i)}$ which implies:
\

\

$\phi(n_i+k(n_i)c(n_i))-\phi(n_i)<k(n_i)\ln(4+2c(n_i))$
\

\

\textbf{Remark.} Notice the difference of this inequality from the similar one in the proof of Lemma 2.3.
\

\

Comparing the highest degrees of both sides implies
 $(1+\epsilon)\mu\leq \mu$ and so we get a contradiction. 
\end{proof}

In order to use Lemma 2.2. we need to make sure that $e^{c(n_i)}<P_{N}$. Since $N>n_i$ it is enough to establish that $c(n_i)<P_{n_i+1}$.
\

Since $P_{n_i+1}>\phi(n_i+1)> (n_i+1)^{\mu}$ and $c(n_i)=n_i^{1-\mu+\epsilon}$ comparing the highest degrees gives: $1-\mu+\epsilon<\mu$ which is true for any $\epsilon$ small enough as $\mu>\sqrt{3}-1$.
\

\

We apply Lemma 2.2 to the sequence of partitions ${\xi}_N, {\xi}_{N+1},\ldots, {\xi}_{N+c(n_i)}$ and find that there are three points $x_p, x_q, x_r$ in a good position with indexes $p, q, r$ in the range 
$[N+1,\ldots,N+c(n_i)]$ and with pairwise distances bounded by 
$d=e^{c(n_i)}/P_{n_i+1}<e^{n_i^{1-\mu+\epsilon}}/e^{N^{\mu}}\leq e^{n_i^{1-\mu+\epsilon}}/e^{{(n_i+1)}^{\mu}}$.
\

\

As in the proof of theorem 4.1. we consider vertices $P, Q, R$ corresponding to the unfolding along generalized diagonals $x_p, x_q, x_r$ and entirely repeating the argument of theorem 4.1. we get that the area of the triangle $PQR$ can be represented as $\frac{A(\alpha,\beta)}{{\sin}^2(\alpha+\beta)}$, where
 $A(\alpha, \beta)$ is a trigonometric polynomial with integer coefficients of degree at most $4c(n_i)$.
\

Moreover, again, by repeating the argument of Theorem 4.1. we have the estimate: $0\neq |A(\alpha, \beta)|<D_{\delta}{\sin}^2(\alpha+\beta)c(n_i)d<e^{-bn_i^{\mu}}$ for some $b>0$.
\

\

Let $\mathcal{G}_n$ be a set of trigonometric polynomials corresponding to unfolding combinatorics of length $c(n)=4n^{1-\mu+\epsilon}$.  Any polynomial is determined by the combinatorics and a choice of vertices. It implies  $|\mathcal{G}_n|< e^{tn^{1-\mu+\epsilon}}$, for some $t>0$.
\

We are now going to estimate the measure of  the following set ${\mathcal{C}}_n$ of triangle angles: 
${\mathcal{C}}_n=\lbrace (\alpha, \beta)| \exists A\in{\mathcal{G}}_n :|A(\alpha,\beta)|<e^{-bn^{\mu}}\rbrace $.
\

\

We now pick $A\in{\mathcal{G}}_n$ and take $m=Fn^{1-\mu+\epsilon}$, where $F>4$ is a large enough constant to be chosen later.
\

By corollary 4.1:  $Leb\lbrace (\alpha, \beta): |A(\alpha, \beta)|<e^{-RF^{2}n^{2(1-\mu+\epsilon)}}\rbrace<e^{-cFn^{(1-\mu+\epsilon)}}$.
\

We need now to show that for large enough $n$: $e^{-bn^{\mu}}<e^{-RF^2n^{2(1-\mu+\epsilon)}}$ .
\

Comparing the degrees we get: $2(1-\mu+\epsilon)<\mu$ which is true for any $\mu>\sqrt{3}-1$ and $\epsilon<0.01$.
\

\

The argument above implies: $Leb({\mathcal{C}}_n)<e^{-cFn^{(1-\mu+\epsilon)}}|{\mathcal{G}}_n|<e^{(t-cF)n^{(1-\mu+\epsilon)}}$
\

\

If we choose $F>t/c$ then $\sum Leb({\mathcal{C}}_n)<\infty$.
\

Let $Y$ be a subset of $X$ such that for any triangle $\Delta\in Y$ there is an infinite subsequence of times $n_i$ such that $n_{i+1}-n_{i}>n_i^{1+\epsilon}$. 
Then $\Delta\in{\mathcal{C}}_{n_i}$ for infinitely many $n_i$ and so by Borel-Cantelly argument $Leb(Y)=0$.
\end{proof}

We fiinally have all the tools to prove the main theorem.
\

\begin{Theorem}[Explicit sub-exponential estimate] For any $\epsilon>0$ and a typical triangle: $P_n<Ce^{n^{\sqrt{3}-1+\epsilon}}$ for some $C>0$.
\end{Theorem}

\begin{proof}Consider a triangle $\Delta\in X\setminus Y$.  It is enough to prove an estimate in case of one vertex. Pick any $\epsilon>0$ small enough and then pick positive $\delta\ll\epsilon$ small enough. Consider a sequence $n_i$ corresponding to $\mu=\sqrt{3}-1+\delta$. By Theorem 4.3 for all $i$ large enough: $n_{i+1}<n_{i}+n_{i}^{1+\delta}$. As we are able to slightly perturb $\delta$ if needed we may assume: $n_{i+1}<n_i^{1+\delta}$.
\

We then pick $n$ large enough and find $i$ such that $n_i\leq n<n_{i+1}$. By monotonicity $P_{n_i}\leq P_n\leq P_{n_{i+1}}$, so
\

\ 

$P_n\leq e^{n_{i+1}^{\sqrt{3}-1+\delta}}\leq e^{n_i^{(1+\delta)(\sqrt{3}-1+\delta)}}\leq e^{n^{(1+\delta)(\sqrt{3}-1+\delta)}}\leq e^{n^{\sqrt{3}-1+\epsilon}}$.
\end{proof}
\

\

\textbf{Remark.} It is a well-known observation (see, say [8]) that the number of different combinatorial  types of periodic orbits of length not greater than $n$ for a polygonal billiard is bounded above by $P_n$ and so as a corollary of Theorem 4.4 we have:
\

\begin{Theorem}[Growth of periodic orbits] For any $\epsilon>0$ and a typical triangle: $Per_n<Ce^{n^{\sqrt{3}-1+\epsilon}}$ for some $C>0$, where $Per_n$ is a number of different combinatorial types of periodic orbits of length not greater than $n$.
\end{Theorem}

\end{document}